\documentclass[11pt,a4paper,leqno]{amsart}

\usepackage[latin1]{inputenc}
\usepackage[T1]{fontenc}
\usepackage{amsfonts}
\usepackage{amsmath}
\usepackage{amssymb}
\usepackage{eurosym}
\usepackage{mathrsfs}
\usepackage{palatino}
\usepackage{color}
\usepackage{esint}
\usepackage{url}
\usepackage{verbatim}
\usepackage{graphicx}

\allowdisplaybreaks[4]

\usepackage{enumerate}

\usepackage[pagebackref,hypertexnames=false, colorlinks, citecolor=blue, linkcolor=blue, urlcolor=red]{hyperref}

\newcommand{\R}{\mathbb{R}}
\newcommand{\C}{\mathbb{C}}

\numberwithin{equation}{section}

\newcommand{\ud}[0]{\,\mathrm{d}}

\newcommand{\BMO}[0]{\operatorname{BMO}}

\newcommand{\sign}[0]{\operatorname{sgn}}

\newcommand{\eps}[0]{\varepsilon}

\newcommand{\wt}[1]{{\widetilde{#1}}}

\swapnumbers
\theoremstyle{plain}
\newtheorem{thm}[equation]{Theorem}
\newtheorem{lem}[equation]{Lemma}

\theoremstyle{definition}

\theoremstyle{remark}
\newtheorem{rem}[equation]{Remark}

\title{Curved commutators in higher dimensions}

 \author{Kangwei Li}
 \author{Yunan Zeng}

 \address[K.L.]{School of Mathematical Sciences, Zhejiang Normal University, Jinhua 321004, China}
 \email{kangwei.li@zjnu.edu.cn}
 \address[Y.Z.]{Center for Applied Mathematics, Tianjin University, Weijin Road 92, 300072 Tianjin, China}
 \email{zynn@tju.edu.cn}

\makeatletter
\@namedef{subjclassname@2020}{%
  \textup{2020} Mathematics Subject Classification}
\makeatother


\begin{document}

\allowdisplaybreaks

\begin{abstract}
We complete the characterization of the $L^p$ boundedness of the commutators $[b, H_\gamma]$
of pointwise multiplication and the Hilbert transform along monomial curves in any dimension. This 
fully solve a key point left open by Bongers,  Guo, Li and Wick, which was affirmatively answered recently only for dimension two. We achieve this by a novel but elementary argument, avoiding the complicated geometry structure in higher dimensions. 
\end{abstract}

\maketitle
%\tableofcontents
\section{Introduction}
We aim to complete the picture of characterizing the $L^p$ boundedness of the commutators $[b, H_\gamma]$
of pointwise multiplication and the Hilbert transform along monomial curves in any dimension. 
Here a monomial curve means 
\begin{equation*}
\gamma(t)= \begin{cases}
(\varepsilon_1 |t|^{\alpha_1}, \ldots,  \varepsilon_n |t|^{\alpha_n})\quad & \text{if $t\ge 0$}\\
(\varepsilon_1' |t|^{\alpha_1}, \ldots, \varepsilon_n'|t|^{\alpha_n})\quad & \text{if $t< 0$},
\end{cases}
\end{equation*}where $0<\alpha_1<\cdots<\alpha_n<\infty$ 
$, 
\varepsilon_i, \varepsilon_i'\in \{-1,1\}
$, and there exists $j\in\{1,\ldots,n\}$ such that $\varepsilon_j\neq \varepsilon_j'$. Then the Hilbert 
transform along $\gamma$ is defined as 
\[
H_\gamma f(x):= p.v. \int_{\mathbb R} f(x-\gamma(t)) \frac{\ud t}{t},\qquad x\in \R^n,
\]
and finally, the commutator $[b, H_\gamma]$ is formulated as
\[
[b, H_\gamma](f)(x)= b(x) H_\gamma f(x)- H_\gamma(bf)(x).
\]

The commutator $[b, H_\gamma]$ was first studied by Bongers,  Guo, Li and Wick in \cite{BGLW}, 
in which they proved that for $1<p<\infty$
\begin{equation}\label{eq:bglw}
\|[b, H_\gamma]\|_{L^p\to L^p}\le C \|b\|_{\BMO_\gamma},
\end{equation}
where 
\[
\|b\|_{\BMO_\gamma}:= \sup_{Q: \, \text{$\gamma$-cube}} \frac 1{|Q|}\int_Q |b-b_Q |,\qquad b_Q:=\frac 1{|Q|}\int_Q b.
\]
Here the supremum is taken over all $\gamma$-cubes, namely, $Q=I_1\times \ldots\times I_n$ with 
\[
\ell(I_1)^{1/{\alpha_1}}=\cdots= \ell(I_n)^{1/{\alpha_n}}.
\]
However, in the reverse direction to \eqref{eq:bglw}, they did not prove the necessity of $\BMO_\gamma$, and thus leaving out an open question. 

In \cite{O22} Oikari first attacked this problem for the parabolic Hilbert transform in dimension two. 
Here the parabolic Hilbert transform means $\gamma(t)=(t, t^2)$. Essentially the argument in \cite{O22} already shed light for general monomial curves in dimension two but assuming the curve to intersect adjacent quadrants of the plane. Such a restriction was removed recently in \cite{LMO} by Martikainen, Oikari and the first named author. The main idea for \cite{O22} and \cite{LMO} 
is a non-trivial adaptation of the approximate weak factorization argument, depending heavily on the geometry structure of the curve. Thus making the higher dimensional case quite challenging. 

The approximate weak factorization method was introduced by Hyt\"onen in \cite{Hyt18}. One of the reasons for him to introduce this technique is that the median method by Lerner, Ombrosi and Rivera-R\'ios \cite{LOR19} has the restriction that $b$ has to be real-valued. Both techniques are widely used in various commutator estimates, see e.g. \cite{ALM, AHLMO, LMV} and the references therein. However, in our setting both the approximate weak factorization and median methods will meet challenges.  We have explained this point above for the  approximate weak factorization method. On the other hand, it is already written in \cite{BGLW} that `\emph{classical techniques such as the median method break down due to the complicated geometry of the
operator}'. 

Our new approach is rather elementary. In fact, for a given $\gamma$-cube $Q$, unlike \cite{O22} and \cite{LMO}, we define its companion $\gamma$-cube $P$ in a uniform way. That is, we do not care the precise values of $\eps_j$ and  $\eps_j'$. On the other hand, we also adopt the merit of the median method. The difference is that instead of using the median of $b$ on $P$, we use directly $b(z)$ when $z$ belongs to some major subset of $P$.

Now we record the main result of this paper as follows.
\begin{thm}\label{thm:main}
Let $b\in L_{\rm{loc}}^1 (\R^n; \C)$ and $H_\gamma$ be the Hilbert transform along a monomial curve $\gamma$. Suppose that $[b, H_\gamma]$ is a bounded operator on $L^p$. Then $b\in \BMO_\gamma(\R^n)$ and moreover, 
\[
\|b\|_{\BMO_\gamma(\R^n)}\lesssim \|[b, H_\gamma]\|_{L^p\to L^p}.
\]
\end{thm}
Corollaries such as the factorization of Hardy spaces, which has applications in PDEs via compensated compactness and div-curl lemmas \cite{CLMS, Lind} as well as the Jacobian problem \cite{Hyt18}, are standard (see e.g. \cite{CRW, O22}) and hence we omit the statements here.

\vspace{0.3cm}

\noindent \textbf{Acknowledgements.}
This work is supported by the National Natural Science Foundation of China through
project numbers 12671125 and 12222114. 

\section{Preliminaries}
To simplify the notation (and also computations) we focus on three dimensions, it will be clear that our strategy can be extended to higher dimensions easily. Thus from now on we consider 
\begin{equation*}
\gamma(t)= \begin{cases}
(\varepsilon_1 |t|^{\alpha_1}, \varepsilon_2 |t|^{\alpha_2},  \varepsilon_3 |t|^{\alpha_3})\quad & \text{if $t\ge 0$}\\
(\varepsilon_1' |t|^{\alpha_1}, \varepsilon_2' |t|^{\alpha_2}, \varepsilon_3'|t|^{\alpha_3})\quad & \text{if $t< 0$},
\end{cases}
\end{equation*}where $0<\alpha_1<\alpha_2<\alpha_3<\infty$ 
$, 
\varepsilon_i, \varepsilon_i'\in \{-1,1\}
$, and there exists $j\in\{1,2,3\}$ such that $\varepsilon_j\neq \varepsilon_j'$. 
Easy change of variable argument (see e.g. \cite{LMO}) allows us to assume that $\alpha_1=1$. So in below we shall work with $(\alpha_1, \alpha_2, \alpha_3):=(1, \beta_1, \beta_2)$ with $1<\beta_1<\beta_2<\infty$. On the other hand, 
recall that in \cite{LMO} and \cite{O22}, the geometry of the curve plays the key role. In particular, different signs $\eps_j, \eps_j'$ in $H_\gamma$ lead to different approximate weak factorization argument. However, in this paper, we find this is actually not necessary. 

Given a $\gamma$-cube $Q$, namely, $Q=I\times J\times K$ with $\ell(J)=\ell(I)^{\beta_1}$ and $\ell(K)=\ell(I)^{\beta_2}$,  we define 
\begin{equation}\label{eq:e1}
P:= Q+  \gamma(A\ell(I))+  \gamma(NA\ell(I))+  \gamma(N^2A \ell(I)),
\end{equation}
where $A, N$ are sufficiently large which will be chosen later.
Note that in this definition we only used the `positive part' of the curve (i.e. $t>0$), and it turns out that this is enough in the lower bound of the commutator.  The key result of this section is the following
\begin{lem}\label{lem:key}
Let $Q$ be a $\gamma$-cube and $P$ be defined as in \eqref{eq:e1} with sufficiently large $A, N$. Then  for any $x\in Q$ and $z\in P$, there is a unique triple $(u_1, u_2, u_3)$ with $$u_i\in ((N^{i-1}A -2)\ell(I), (N^{i-1}A +2)\ell(I))=:I_A^i$$ such that $x+  \gamma(u_1)+  \gamma(u_2)+  \gamma(u_3)=z$. 
\end{lem}
\begin{proof}
For any $x\in Q=I\times J\times K$, set $$x+ \gamma(A\ell(I))+   \gamma(NA \ell(I))+  \gamma(N^2 A\ell(I))=:x'\in P.$$
For simplicity we denote $a_j:=N^{j-1}A  \ell(I)$. 
Note that then $x+  \gamma(u_1)+  \gamma(u_2)+  \gamma(u_3)=z$ holds   if and only if 
\[
z-x'=   (\gamma(u_1)- \gamma(a_1))+  (\gamma(u_2)-\gamma(a_2))+  (\gamma(u_3)- \gamma(a_3)).
\]
Now let us restrict the equation to the region $u_j \in I_A^j$ and 
set $v_j= u_j-a_j\in (-2\ell(I), 2\ell(I))$, componentwise this and using Taylor's formula we have 
\begin{equation}
\left\{
\begin{aligned}
		z_1-x_1' &=\eps_1 v_1+ \eps_1 v_2+ \eps_1 v_3,  \\
		z_2-x_2'&= \eps_2 \sum_{j=1}^3 \Big(\beta_1 a_j^{\beta_1-1}  v_j +
		\int_{a_j}^{a_j+v_j} (a_j+v_j-t)  \beta_1 (\beta_1-1) t^{\beta_1-2}\ud t\Big),  \\
		z_3-x_3'&=\eps_3 \sum_{j=1}^3 \Big(\beta_2 a_j^{\beta_2-1}  v_j +\int_{a_j}^{a_j+v_j} (a_j+v_j-t)  \beta_2 (\beta_2-1) t^{\beta_2-2}\ud t\Big). 
		\end{aligned}
		\right.
	\end{equation}
Now define 
	\[
	\renewcommand{\arraystretch}{1.5}
	M:=
	\begin{pmatrix}
	1& 1& 1\\
	\beta_1 a_1^{\beta_1-1}  & \beta_1 a_2^{\beta_1-1} &\beta_1 a_3^{\beta_1-1} \\
	 \beta_2 a_1^{\beta_2-1}  & \beta_2 a_2^{\beta_2-1} &\beta_2 a_3^{\beta_2-1}
	\end{pmatrix}, \qquad 
	 v:= 
	\begin{pmatrix}
	v_1\\
	v_2\\
	v_3
	\end{pmatrix}.
	\]
and 
	$
	T(  v ):= M^{-1}(q(v )), 
$
	where 
	\[
	\renewcommand{\arraystretch}{1.5}
	q(v ):= 
	\begin{pmatrix}
	\eps_1^{-1} (z_1-x_1')\\
	\eps_2^{-1} (z_2-x_2') -h_2(v )\\
	\eps_3^{-1} (z_3-x_3') -h_3(v )
	\end{pmatrix}
	=:
	\begin{pmatrix}
	q_1\\
	q_2\\
	q_3
	\end{pmatrix}
	\]
	with 
	\[
	h_j(v ):= \sum_{i=1}^3\int_{a_i}^{a_i+v_i} (a_i+v_i-t)  \beta_{j-1} (\beta_{j-1}-1) t^{\beta_{j-1}-2}\ud t,\qquad j=2,3.
	\]
	The problem is then reduced to proving that there exists a unique $v $ with $|v_i|\le 2\ell(I)$
	such that $Tv =v $. Direct computation gives that the inverse matrix $M^{-1}$ equals
	\[
	\renewcommand{\arraystretch}{1.5}
	\frac 1{\det M}\begin{pmatrix}
	\beta_1\beta_2 (a_2^{\beta_1-1}a_3^{\beta_2-1}- a_3^{\beta_1-1}a_2^{\beta_2-1}) & \beta_2 (a_2^{\beta_2-1}-a_3^{\beta_2-1})& \beta_1 (a_3^{\beta_1-1}-a_2^{\beta_1-1})\\
	\beta_1\beta_2 (a_3^{\beta_1-1}a_1^{\beta_2-1}- a_1^{\beta_1-1}a_3^{\beta_2-1}) & \beta_2 (a_3^{\beta_2-1}-a_1^{\beta_2-1})& \beta_1 (a_1^{\beta_1-1}-a_3^{\beta_1-1})  \\
	\beta_1\beta_2 (a_1^{\beta_1-1}a_2^{\beta_2-1}- a_2^{\beta_1-1}a_1^{\beta_2-1}) & \beta_2 (a_1^{\beta_2-1}-a_2^{\beta_2-1})& \beta_1 (a_2^{\beta_1-1}-a_1^{\beta_1-1}) 
	\end{pmatrix},
	\]
	where 
	\begin{align*}
	\det M&=\beta_1 \beta_2 \big((a_2^{\beta_1-1}a_3^{\beta_2-1}-a_2^{\beta_2-1}a_3^{\beta_1-1})- (a_1^{\beta_1-1}a_3^{\beta_2-1}-a_1^{\beta_2-1}a_3^{\beta_1-1})\\
	&\hspace{4cm}+
	(a_1^{\beta_1-1}a_2^{\beta_2-1}-a_1^{\beta_2-1}a_2^{\beta_1-1})\big). 
	\end{align*}
Let  $N$ be sufficiently large (depending only on $\beta_1, \beta_2$), we may simply estimate as 
\begin{align*}
\det M \ge \beta_1 \beta_2 \bigg(a_2^{\beta_1-1}a_3^{\beta_2-1}- \sum_{\substack{1\le i, j\le 3\\
i\ne j\\ (i,j)\ne (2,3)}} a_i^{\beta_1-1}a_j^{\beta_2-1}\bigg)\ge \frac {9}{10}\beta_1 \beta_2a_2^{\beta_1-1}a_3^{\beta_2-1}.
\end{align*}
On the other hand, for $v$ with $|v_i|\le 2\ell(I)$, we have  $|q_1|\le \ell(I)$, 
\[
|q_2|\le \ell(I)^{\beta_1}(1+  12\beta_1(\beta_1-1)  \max\{ (N^2 A +2)^{\beta_1-2}, (A-2)^{\beta_1-2}\})
\]
and
\[
|q_3|\le \ell(I)^{\beta_2}(1 +12\beta_2(\beta_2-1)  \max\{ (N^2 A +2)^{\beta_2-2}, (A-2)^{\beta_2-2}\}),
\]
where e.g. we controlled the integral in $q_2$ as 
\begin{align*}
&\Big|\int_{a_j}^{a_j+v_j} (a_j+v_j-t)  \beta_1 (\beta_1-1) t^{\beta_1-2}\ud t\Big|\\
&\le\beta_1 (\beta_1-1) \Big|\int_{a_j}^{a_j+v_j} |v_j|  \max\{ (N^2 A +2)^{\beta_1-2}, (A-2)^{\beta_1-2}\} \ell(I)^{\beta_1-2}\ud t\Big|\\
&\le 4\beta_1 (\beta_1-1)\ell(I)^{\beta_1}\max\{ (N^2 A +2)^{\beta_1-2}, (A-2)^{\beta_1-2}\}.
\end{align*}
It follows that if we take $N$ as above and let $A$ be sufficiently large then 
\begin{align*}
|(T(v))_i|\le 2\ell(I),\qquad 1\le i\le3,
\end{align*}where $(T(v))_i$ stands for the $i$-th coordinate of $T(v)$.
This means that the map $T$ sends $[-2\ell(I), 2\ell(I)]^3$ into $[-2\ell(I), 2\ell(I)]^3$. Moreover, for 
$v, \wt v\in [-2\ell(I), 2\ell(I)]^3$, since 
\begin{align*}
T(v)-T(\wt v)= M^{-1} \begin{pmatrix}
	0\\
	h_2(\wt v) -h_2(v )\\
	h_3(\wt v) -h_3(v )
	\end{pmatrix}
\end{align*}
and 
\begin{align*}
 h_j(\wt v) -h_j(v ) 
&=\sum_{i=1}^3\Big(\int_{a_i}^{a_i+\wt v_j} (\wt v_i-v_i) \beta_{j-1} (\beta_{j-1}-1) t^{\beta_{j-1}-2}\ud t\\
&\qquad+  \int_{a_i+v_i}^{a_i+\wt v_i} (a_i+v_i-t) \beta_{j-1} (\beta_{j-1}-1) t^{\beta_{j-1}-2}\ud t\Big)
\end{align*}
then similar as above, when taking $N$ and $A$ to be sufficiently large, we have  
\[
|T(v)-T(\wt v)|\le \frac 12 |v-\wt v|.
\]
Here what is crucial is that, the power of $A$ for the downstairs (i.e. $\det M$) is $A^{\beta_1+\beta_2-2}$, while for the upstairs it is $A^{\beta_1+\beta_2-3}$. The proof is complete by the compressed mapping principle. 
\end{proof}

\begin{rem}
We remark that it is clear our proof for Lemma \ref{lem:key}  can be easily extended to higher dimensions. The key is that the product of the diagonal of the matrix is the major term of its determinant. 
\end{rem}

\section{Proof of the main result}
 Now we are ready to prove  Theorem \ref{thm:main}. Let $Q$ be a fixed $\gamma$-cube, and set $P$ as in \eqref{eq:e1}. Then we may define the exceptional set 
 \[
 E:=\big\{z\in P: |b(z)-b_Q|\ge \langle |b-b_Q|\rangle_Q\big\},\qquad b_Q:=\langle b\rangle_Q.
 \]
 We shall split the argument into two cases: $|E|\ge |P|/2$ or $|E|<|P|/2$.
 \subsection{The case $|E|\ge |P|/2$}
This is the nicer case since we already controlled the mean oscillation by $|b(z)-b_Q|$ on a major subset of $P$.
Note that we can write 
\[
|b(z)-b_Q|= \frac 1{|Q|}\Big| \int_Q (b(z)-b(x))\ud x\Big|.
\]
The problem is then reduced to relating the integral on the right hand side with the Hilbert transform along curves. Fix $z\in E$,
 having Lemma \ref{lem:key} at hand, there is a unique $(u_1, u_2, u_3)$ such that 
 \[
 x=z-\gamma(u_1)-\gamma(u_2)-\gamma(u_3).
 \]
 Then formally we may write 
  \begin{align*}
 b(z)-b(x)&=\big[b(z)-b(z-\gamma(u_3))\big]+ \big[b(z-\gamma(u_3))- b(z-\gamma(u_2)-\gamma(u_3))\big]\\
 &\qquad+\big[b(z- \gamma(u_2)-\gamma(u_3))-b(z-\gamma(u_1)-\gamma(u_2)-\gamma(u_3))\big].
 \end{align*}
Now we have seen the commutator structure already appears, so then it is natural to perform the change of variable argument. To proceed, 
let 
\[
U:= \big\{(u_1, u_2, u_3): u_i\in I_A^i\big\}.
\]
 We may define 
\[
D_z:=\big\{(u_1, u_2, u_3)\in U: z-\gamma(u_1)-\gamma(u_2)-\gamma(u_3)\in Q\}. 
\]
Then clearly we have 
\begin{equation}\label{eq:trans}
\int_Q (b(z)-b(x))\ud x= 
\int_{D_z}\big(b(z)-b(z-\gamma(u_1)-\gamma(u_2)-\gamma(u_3))\big)|\det J(u)| \ud u_1 \ud u_2 \ud u_3,
\end{equation}
where $J(u)$ is the Jacobian matrix 
\[
	\renewcommand{\arraystretch}{1.5}
	J(u):=
	-\begin{pmatrix}
	\eps_1 & \eps_1& \eps_1\\
	\eps_2\beta_1 u_1^{\beta_1-1}  & \eps_2\beta_1 u_2^{\beta_1-1} &\eps_2\beta_1 u_3^{\beta_1-1} \\
	\eps_3 \beta_2 u_1^{\beta_2-1}  & \eps_3\beta_2 u_2^{\beta_2-1} &\eps_3\beta_2 u_3^{\beta_2-1}
	\end{pmatrix}.
	\]
	Now write  
	\begin{align}\label{eq:e100}
	\begin{split}
	|\det J(u)|&= N^3A^3 \ell(I)^3|\det J(A\ell(I), NA\ell(I), N^2 A\ell(I))|\\
	&\qquad \cdot\frac{|\det J(u)|u_1 u_2 u_3}{ N^3A^3 \ell(I)^3|\det J(A\ell(I), NA\ell(I), N^2 A\ell(I))|}\cdot \frac 1{u_1 u_2 u_3}\\
	&= N^3A^3 \ell(I)^3|\det J(A\ell(I), NA\ell(I), N^2 A\ell(I))|\cdot \bigg( \frac 1{u_1 u_2 u_3}\\
	&\qquad+ \Big(\frac{|\det J(u)|u_1 u_2 u_3}{ N^3A^3 \ell(I)^3|\det J(A\ell(I), NA\ell(I), N^2 A\ell(I))|}-1\Big)\frac 1{u_1 u_2 u_3}\bigg).
	\end{split}
	\end{align}
	For simplicity we denote $$C(N,A,\ell(I))= N^3A^3 \ell(I)^3|\det J(A\ell(I), NA\ell(I), N^2 A\ell(I))|.$$
	Then substitute \eqref{eq:e100} into \eqref{eq:trans}, we see that LHS of \eqref{eq:trans} can be written as the sum of two terms, one of which is
	\begin{align*}
&C(N,A,\ell(I))\int_{D_z}\big(b(z)-b(z-\gamma(u_1)-\gamma(u_2)-\gamma(u_3))\big)\frac{\ud u_1\ud u_2 \ud u_3}{u_1 u_2 u_3}\\
&=C(N,A,\ell(I))\int_{D_z}\big(b(z)-b(z- \gamma(u_3))\big)\frac{\ud u_1\ud u_2 \ud u_3}{u_1 u_2 u_3}\\
&\quad+ C(N,A,\ell(I))\int_{D_z}\big(b(z-\gamma(u_3))-b(z-\gamma(u_2)- \gamma(u_3))\big)\frac{\ud u_1\ud u_2 \ud u_3}{u_1 u_2 u_3}\\
&\quad+C(N,A,\ell(I))\int_{D_z}\big(b(z-\gamma(u_2)- \gamma(u_3))-b(z-\gamma(u_1)-\gamma(u_2)-\gamma(u_3))\big)\frac{\ud u_1\ud u_2 \ud u_3}{u_1 u_2 u_3}\\
&=: I_1(z)+I_2(z)+I_3(z).
	\end{align*}
	The other term is 
	\begin{align*}
	 &C(N,A,\ell(I))\int_{D_z}\big(b(z)-b(z-\gamma(u_1)-\gamma(u_2)-\gamma(u_3))\big)  \Big(\frac{\det J(u) u_1u_2u_3}{C(N,A,\ell(I))}-1\Big)\frac{\ud u_1\ud u_2 \ud u_3}{u_1 u_2 u_3}\\
	 &=: I_4(z).
	\end{align*}
We deal with $I_4(z)$ first. The key observation is that when $(u_1, u_2, u_3)\in D_z$, we have the following
\[
\lim_{A\to \infty}\sup_{u\in D_z}\left|\frac{\det J(u) u_1u_2u_3}{C(N,A,\ell(I))}-1\right|\cdot \frac{C(N,A,\ell(I))}{|\det J(u)|u_1u_2u_3}=0,
\]
where we have fixed $N$ to be sufficiently large (depending only on $\beta_1, \beta_2$). Then we may choose $A$ to be sufficiently large so that 
\begin{equation}\label{eq:eq103}
\sup_{u\in D_z}\left|\frac{\det J(u) u_1u_2u_3}{C(N,A,\ell(I))}-1\right|\cdot \frac{|C(N,A,\ell(I))|}{|\det J(u)|u_1u_2u_3}< \frac 18
\end{equation}
and we have
\begin{align*}
|I_4(z)|&\le \frac 18 \int_{D_z}| b(z)-b(z-\gamma(u_1)-\gamma(u_2)-\gamma(u_3))| \cdot |\det J(u)| \ud u_1 \ud u_2 \ud u_3\\
&= \frac 18\int_Q |b(z)-b(x)|\ud x
\le \frac 18 |b(z)- b_Q| |Q|+ \frac 18\langle | b-b_Q|\rangle_Q |Q|\\
&\le \frac 12 |b(z)- b_Q| |Q|,
\end{align*}
where in the last step we have used the definition of $E$. As a summary of the above arguments, we have arrived at 
\begin{equation}\label{eq:eq101}
|b(z)- b_Q|  \le \frac 2{|Q|} (I_1(z)+I_2(z)+I_3(z)),\qquad z\in E.
\end{equation}
Next we move to estimate $I_1, I_2$ and $I_3$, these terms will actually be estimated in a similar way. Indeed, we have 
\begin{align*}
I_1(z)&= C(N,A,\ell(I))\int_{I_A^1}\int_{I_A^2}\int \big(b(z)-b(z- \gamma(u_3))\big)1_{Q+\gamma(u_1)+\gamma(u_2)}(z-\gamma(u_3))\frac{ \ud u_3\ud u_1\ud u_2}{u_3u_1 u_2 }\\
&= C(N,A,\ell(I))\int_{I_A^1}\int_{I_A^2}[b, H_\gamma](1_{Q+\gamma(u_1)+\gamma(u_2)} )(z)\frac{\ud u_1 \ud u_2}{u_1 u_2}.
\end{align*}
Similarly, 
\begin{align*}
I_2(z)&=C(N,A,\ell(I))\int_{I_A^1}\int_{I_A^3}[b, H_\gamma](1_{Q+\gamma(u_1)} )(z-\gamma(u_3))\frac{\ud u_1 \ud u_3}{u_1 u_3},\\
I_3(z)&=C(N,A,\ell(I))\int_{I_A^2}\int_{I_A^3}[b, H_\gamma](1_{Q} )(z-\gamma(u_2)-\gamma(u_3))\frac{\ud u_2 \ud u_3}{u_2 u_3}.
\end{align*}
Taking the $L^p(E)$ norm on both sides of \eqref{eq:eq101}, and applying the Minkowski's inequality we get 
\begin{align*}
\langle |b-b_Q|\rangle_Q & |E|^{1/p}\le \|b-b_Q\|_{L^p(E)}\\
&\le \frac2{|Q|}C(N,A,\ell(I))\bigg(\int_{I_A^1}\int_{I_A^2}\| [b, H_\gamma](1_{Q+\gamma(u_1)+\gamma(u_2)} )\|_{L^p} \frac{\ud u_1 \ud u_2}{u_1 u_2}\\
& + \int_{I_A^1}\int_{I_A^3}\| [b, H_\gamma](1_{Q+\gamma(u_1)} )\|_{L^p}\frac{\ud u_1 \ud u_3}{u_1 u_3} +\int_{I_A^2}\int_{I_A^3}\| [b, H_\gamma](1_{Q} )\|_{L^p}\frac{\ud u_2 \ud u_3}{u_2 u_3}\bigg)\\
&\lesssim_{A, \beta_1, \beta_2}  \| [b, H_\gamma]\|_{L^p\to L^p}|Q|^{1/p}.
\end{align*}
Hence 
\[
\langle |b-b_Q|\rangle_Q \lesssim_{A, \beta_1, \beta_2}  \| [b, H_\gamma]\|_{L^p\to L^p}.
\]
This completes the proof of the case $|E|\ge |P|/2$.
\subsection{The case $|E|< |P|/2$}
In this case we have $|P\setminus E|>|P|/2$, where 
\[
P\setminus E=\big\{z\in P: |b(z)-b_Q|< \langle |b-b_Q|\rangle_Q\big\}.
\]
Then we may write 
\begin{equation}\label{eq:eq104}
\begin{split}
 \langle |b-b_Q|\rangle_Q&= \frac 1{|Q|}\int_Q (b(x)-b_Q) \sign (b(x)-b_Q)\\
 &= \frac 1{|Q|}\int_Q (b(x)-b_Q) (\sign (b(x)-b_Q)- \langle \sign (b(x)-b_Q)\rangle_Q )\\
 &=  \frac 1{|Q|}\int_Q  (b(x)-b(z)) (\sign (b(x)-b_Q)- \langle \sign (b(x)-b_Q)\rangle_Q )\\
 &=: \frac 1{|Q|}\int_Q  (b(x)-b(z)) g(x) \ud x,
 \end{split}
\end{equation}
where $z\in P\setminus E$ and we have used the cancellation property of $(b-b_Q)1_Q$ and 
\[
g(x)= (\sign (b(x)-b_Q)- \langle \sign (b(x)-b_Q)\rangle_Q )1_Q(x).
\]
Now we are in a similar situation as the previous case, the difference is instead of $1_Q$ we have $g(x)$ in the above.
Thus, as before we write 
\begin{equation*} 
\begin{split}
&\int_Q (b(x)-b(z))g(x)\ud x\\
&= 
\int_{D_z}\big(b(z-\gamma(u_1)-\gamma(u_2)-\gamma(u_3))-b(z)\big)    g(z-\gamma(u_1)-\gamma(u_2)-\gamma(u_3))|\det J(u)| \ud u \\
&= C(N,A,\ell(I))\int_{D_z}\big(b(z-\gamma(u_1)-\gamma(u_2)-\gamma(u_3))-b(z)\big)\\
&\hspace{3cm}   g(z-\gamma(u_1)-\gamma(u_2)-\gamma(u_3)) \frac{\ud u_1 \ud u_2 \ud u_3}{u_1u_2u_3}\\
&\quad+C(N,A,\ell(I))\int_{D_z}\big(b(z-\gamma(u_1)-\gamma(u_2)-\gamma(u_3))-b(z)\big)\\
&\hspace{3cm}   g(z-\gamma(u_1)-\gamma(u_2)-\gamma(u_3) ) \Big(\frac{\det J(u) u_1u_2u_3}{C(N,A,\ell(I))}-1\Big)\frac{\ud u_1 \ud u_2 \ud u_3}{u_1u_2u_3}\\
&=:II_1(z)+II_2(z).
\end{split}
\end{equation*}
The term $II_1(z)$ will be handled exactly in the same way as the previous case. Indeed, in the same way we can write 
\begin{align*}
II_1&=-  C(N,A,\ell(I))\bigg( \int_{I_A^1}\int_{I_A^2}[b, H_\gamma](  g(\cdot- \gamma(u_1)-\gamma(u_2)))(z)\frac{\ud u_1 \ud u_2}{u_1 u_2}\\
&\hspace{2cm}+ \int_{I_A^1}\int_{I_A^3}[b, H_\gamma](g(\cdot- \gamma(u_1)) )(z-\gamma(u_3))\frac{\ud u_1 \ud u_3}{u_1 u_3}\\
&\hspace{3cm}+\int_{I_A^2}\int_{I_A^3}[b, H_\gamma](g )(z-\gamma(u_2)-\gamma(u_3))\frac{\ud u_2 \ud u_3}{u_2 u_3}.
\end{align*}
Turn to the term $II_2(z)$, again by taking $A$ to be large enough so that \eqref{eq:eq103} holds, then since 
\[
|g(z-\gamma(u_1)-\gamma(u_2)-\gamma(u_3))|\le 2
\]we have 
\begin{align*}
|II_2|&\le \frac 14\int_{D_z}| b(z)-b(z-\gamma(u_1)-\gamma(u_2)-\gamma(u_3))| \cdot |\det J(u)| \ud u_1 \ud u_2 \ud u_3\\
&= \frac 14\int_Q |b(z)-b(x)|\ud x
\le \frac 14 |b(z)- b_Q| |Q|+ \frac 14\langle | b-b_Q|\rangle_Q |Q|\\
&\le \frac 12\langle | b-b_Q|\rangle_Q |Q|.
\end{align*}
Substitute the above estimates into \eqref{eq:eq104} we obtain 
\[
\langle | b-b_Q|\rangle_Q\le \frac 2{|Q|} |II_1(z)|.
\]
The proof is completed if we take the $L^p(P\setminus E)$ norm on both sides.


\begin{thebibliography}{00}
  \bibitem{ALM}
 E. Airta, K. Li, H. Martikainen, Zygmund dilations: bilinear analysis and commutator estimates. Trans. Amer. Math. Soc. 378 (2025), no. 7, 4581--4625.
  
  \bibitem{AHLMO}
  E. Airta, T. Hyt\"onen, K. Li, H. Martikainen, T. Oikari, 
  Off-diagonal estimates for bi-commutators. Int. Math. Res. Not. IMRN 2022, no. 23, 18766--18832. 
  
  \bibitem{BGLW}
 T. Bongers, Z. Guo, J. Li, B.D. Wick, Commutators of Hilbert transforms along monomial curves,
  Studia Math. 257 (2021), no. 3, 295--311.
  
  \bibitem{CLMS}
  R. Coifman, P.-L. Lions, Y. Meyer, S. Semmes, Compensated compactness and Hardy spaces, J.
Math. Pures Appl. 9 (1993) 247--286.

\bibitem{CRW}
R. Coifman, R. Rochberg, G. Weiss, Factorization theorems for Hardy spaces in several variables,
Ann. Math. (2) 103(1976) 611--635.


  
  \bibitem{Hyt18}
T. Hyt\"onen,   The $L^p$-to-$L^q$ boundedness of commutators with applications to the Jacobian operator. J. Math. Pures Appl. (9) 156 (2021), 351--391.

\bibitem{LOR19}
A. Lerner, S. Ombrosi, I.P. Rivera-R\'ios, 
Commutators of singular integrals revisited. Bull. Lond. Math. Soc. 51 (2019), no. 1, 107--119.

  
\bibitem{LMO}
K. Li, H. Martikainen, T. Oikari, Curved commutators in the plane, Math. Ann. 395 (2026), no. 3, Paper No. 66, 36 pp.

\bibitem{LMV}
K. Li, H. Martikainen, E. Vuorinen, Bilinear Calder\'on-Zygmund theory on product spaces. J. Math. Pures Appl. (9) 138 (2020), 356--412.

\bibitem{Lind}
S. Lindberg, On the Hardy space theory of compensated compactness quantities, Arch. Ration.
Mech. Anal. 224 (2017) 709--742.

\bibitem{O22}
T. Oikari, Lower bound of the parabolic Hilbert commutator, Adv. Math. 404 (2022), 108451.

\end{thebibliography}
\end{document}